\documentclass[12pt,final]{article}

\usepackage[english]{babel}
\usepackage{amssymb,amsfonts}
\usepackage{amsmath, amsthm}
\usepackage{tikz-cd}
\usepackage[utf8]{inputenc} 
\usepackage{hyperref}       
\usepackage{url}            
\usepackage{booktabs}       
\usepackage{amsfonts}       
\usepackage{nicefrac}       
\usepackage{microtype}      
\usepackage{lipsum}
\usepackage{graphicx}
\usepackage{amssymb,amsfonts}
\usepackage{amsmath, amsthm}
\usepackage{clock}

\usepackage{euscript}

\theoremstyle{definition}
\newtheorem*{defn*}{Definition}

\theoremstyle{plain}
\newtheorem*{thm*}{Theorem}
\newtheorem{theorem}{Theorem}

\newtheorem*{lm*}{Lemma}

\newtheorem*{prop*}{Proposition}

\theoremstyle{definition}
\newtheorem*{crl*}{Corollary}

\newtheorem*{remark*}{Remark}

\newtheorem*{exm*}{Example}

\title{Singularities of diagonals of Laurent series\\ for rational functions}
\author{Dmitriy Pochekutov}
\date{}

\begin{document}

\maketitle
\begin{abstract}
	We study the complete diagonal of the Laurent series expansion of a rational function
	in $n$-complex variables.  For a denominator that is nondegenerate for its Newton polyhedron, 
	we prove that the complete diagonal, initially defined in a logarithmically convex domain, can be analytically continued along any path in the $r$-dimensional complex torus that avoids an explicitly defined complex analytic set~$L$ called the Landau variety. This variety is constructed as the union of discriminants associated with specific truncations of the denominator to the faces of its Newton polyhedron.	
\end{abstract}

\section{Introduction}
Let $T^n_{\boldsymbol{z}}:=(\mathbb{C}\setminus{0})^n$ be the $n$-dimensional complex torus with coordinates $\boldsymbol{z}:=(z_1,\ldots, z_n)$.
Consider a Laurent polynomial
$$
    f(\boldsymbol{z}):=\sum_{\boldsymbol{\alpha}\in A} a_{\boldsymbol{\alpha}} \boldsymbol{z}^{\boldsymbol{\alpha}},
$$
that is a finite linear combination of monomials $\boldsymbol{z}^{\boldsymbol{\alpha}}:=z_1^{\alpha_1}\cdots z_n^{\alpha_n}$,
whose degrees $\boldsymbol{\alpha}:=(\alpha_1,\ldots, \alpha_n)$ range over a finite subset $A$ of the integer lattice $\mathbb{Z}^n$.
We assume that the coefficients $a_{\boldsymbol{\alpha}}$ are nonzero for all $\boldsymbol{\alpha}\in A$.
Then the Newton polyhedron $\Delta_f$ for $f$ is the convex hull in $\mathbb{R}^n$ of the set $A$. The Laurent polynomial $f$ naturally defines a complex-valued function on $T^n_{\boldsymbol{z}}$. We denote by $Z^\times(f)$
the set of its zeros in $T^n_{\boldsymbol{z}}$.

We are interested in all possible Laurent expansions
\begin{equation}
\label{eq:s0:1}
F(\boldsymbol{z}):= \sum_{\boldsymbol{\alpha}\in \mathbb{Z}^n} c_{\boldsymbol{\alpha}} \boldsymbol{z}^{\boldsymbol{\alpha}}
\end{equation}
centered at the point $\boldsymbol{0}:=(0,\ldots,0)$ for the rational function $g(\boldsymbol{z})/f(\boldsymbol{z})$, where $g(\boldsymbol{z})$ is a Laurent polynomial that is coprime with $f$. It is well known that the set of such expansions is in one-to-one correspondence with the set $\{E\}$ of connected components of the complement $\mathcal{A}_f^c:=\mathbb{R}^n\setminus \mathcal{A}_f$, where $\mathcal{A}_f$ denotes the amoeba of the Laurent polynomial $f$. Recall that the \textit{amoeba} $\mathcal{A}_f:=\Lambda(Z^\times(f))$, where the mapping $\Lambda:T^n_{\boldsymbol{z}} \to \mathbb{R}^n$ is defined by the formula
$$
    \Lambda(\boldsymbol{z}):=(\log|z_1|,\ldots,\log|z_n|).
$$
According to Corollary~1.6 from~\cite[chapter 6]{Gelfand94}, the set~$\{E\}$ consists of convex components~$E$ whose preimages $\Lambda^{-1}(E)$ are precisely the domains of absolute convergence of the expansions~\eqref{eq:s0:1}. Moreover, it follows from the results of \cite[s.~2]{Forsberg00} that to a component~$E$ of the complement $\mathcal{A}_f^c$, one can bijectively assign an integer point~$\boldsymbol{\nu}$ from the Newton polyhedron~$\Delta_f$, and this assignment possesses the remarkable property: the recession cone of the convex set~$E$ in~$\mathbb{R}^n$ coincides with the dual cone to $\Delta_f$ at the point~$\boldsymbol{\nu}$ (for clarity, the reader may refer to Figure~2). The value~$\boldsymbol{\nu}$ assigned to the component~$E$ is called its \textit{order}.

Let  $\boldsymbol{Q}:=(\boldsymbol{q}_1,\ldots,\boldsymbol{q}_r)$ be an $r$-tuple that generates an $r$-dimensional saturated sublattice of the lattice $\mathbb{Z}^n$.
Then the Laurent series
\begin{equation}
\label{eq:s0:2}
    d_{\boldsymbol{Q}}(\boldsymbol{t}):=\sum_{\boldsymbol{k}\in \mathbb{Z}^r} c_{\boldsymbol{Q}\cdot \boldsymbol{k}} \boldsymbol{t}^{\boldsymbol{k}}
\end{equation}
in  variables $\boldsymbol{t}:=(t_1,\ldots,t_r),$ where $\boldsymbol{Q}\cdot \boldsymbol{k} := \boldsymbol{q}_1 k_1+\ldots+\boldsymbol{q}_r k_r,$
we will call the {\it complete $\boldsymbol{Q}$-diagonal of rank~$r$} or {\it corank $s:=n-r$} of the Laurent series~\eqref{eq:s0:1}. For a convergent expansion~\eqref{eq:s0:1}, its $\boldsymbol{Q}$-diagonal has
a non-empty logarithmically convex domain of convergence~$\mathfrak{T}\subset T^r_{\boldsymbol{t}}$.

This construction leads to a class of functions widely represented in enumerative combinatorics~\cite[chapter~3]{Melczer21},~\cite[chapter~6]{Stanley99}, the theory of difference equations~\cite{Leinartas22}, statistical physics~\cite{Bostan13}, theoretical physics~\cite{Batyrev10}, and even group theory~\cite{Bishop24}.
The goal of the present paper is to describe the singularities of the complete $\boldsymbol{Q}$-diagonal of rank~$r$ of the Laurent expansion~\eqref{eq:s0:1} for the rational function $g(\boldsymbol{z})/f(\boldsymbol{z})$ in terms of its denominator. This is a necessary step in solving the question of the transcendence or algebraicity of the diagonal for~$n\geq 3$. In the case of two variables, the complete diagonal of the Laurent series of a rational function is always algebraic~\cite{Pochekutov09}.

The main result of the paper is Theorem~1, formulated in the second section. It is a multidimensional analogue of Theorem~2 from~\cite{Safonov84} that describes the singularities of the diagonal of the Taylor series of a rational function of two variables. Unlike the authors of the above-mentioned paper, who considered an arbitrary polynomial in two variables with a nonzero constant term as the denominator of the rational function, we restrict ourselves to a Laurent polynomial that is nondegenerate for its Newton polyhedron. Note that almost all Laurent polynomials with a fixed Newton polyhedron are nondegenerate for it.

In the third section of the present paper, we recall the known constructions of integral representations for the complete diagonals of Laurent series of rational functions. The proof of Theorem~1 is presented in the fifth section; it is based on the general idea of analytic continuation of functions that admit a representation as integrals with parameters, detailed in~\cite{Pham11, Vasiliev02} and, in a less rigorous but more accessible form, in~\cite[chapter~2]{Savin17}. According to this idea, the value of the integral depends analytically on the parameters outside a complex analytic subset~$L$ in~$T^r_{\boldsymbol{t}}$, traditionally called the {\it Landau variety} (see the fourth section of the present paper).

\section{Formulation of the Main Result}
Recall that a Laurent polynomial~$f$ is {\it nondegenerate for its Newton polyhedron}~$\Delta_f$ if for any face~$\delta$ of the polyhedron~$\Delta_f$, the truncation $f_\delta$ of this polynomial to the face~$\delta$ satisfies the condition:
$$ \nabla_{\boldsymbol{z}} f_\delta (\boldsymbol{p}):=
	\left(\frac{\partial f_\delta}{\partial z_1}(\boldsymbol{p}),\ldots ,\frac{\partial f_\delta}{\partial z_n}(\boldsymbol{p}) \right)
\neq \boldsymbol{0} \text{ for } \boldsymbol{p}\in Z^{\times}(f_\delta).$$

For each truncation $f_\delta$, using the $(r+1)\times n$ matrix
$$
	\boldsymbol{M}_{\delta}(\boldsymbol{z}):=\left(
	\begin{array}{ccc}
	z_1 \frac{\partial f_\delta}{\partial z_1} & \ldots & z_n\frac{\partial f_\delta}{\partial z_n} \\
	\boldsymbol{q}_1^1 & \ldots & \boldsymbol{q}_1^n \\
	\vdots & \ddots & \vdots \\ 
	\boldsymbol{q}_r^1 & \ldots & \boldsymbol{q}_r^n\\
	\end{array}
	\right)
$$
we define a subset $L_\delta$ in $T^r_{\boldsymbol{t}}$ as the image under the monomial mapping
$$
	\boldsymbol{z}\mapsto \boldsymbol{z}^{\boldsymbol{Q}}:=(\boldsymbol{z}^{\boldsymbol{q}_1},\ldots, \boldsymbol{z}^{\boldsymbol{q}_r})
$$
of the set of all points $\boldsymbol{p}\in Z^{\times}(f_\delta)$ where there rank $M_{\delta}(\boldsymbol{p})$ is less than $r+1.$ 

Note that the first row of the matrix $M_{\delta}(\boldsymbol{z})$ is the value 
$$
	\gamma_f(\boldsymbol{z}):=\left(z_1 \frac{\partial f_\delta}{\partial z_1}  : \ldots : z_n \frac{\partial f_\delta}{\partial z_n} \right)
$$
of the \textit{logarithmic Gauss map} $\gamma_{f_\delta}: Z^{\times}(f_\delta)\to \mathbb{CP}^{n-1}$ of the hypersurface~$ Z^{\times}(f_{\delta})$. Denoting by $\mathcal{L}(\boldsymbol{q}_1,\ldots, \boldsymbol{q}_r)$ the linear span over $\mathbb{C}$ of the vectors
$\boldsymbol{q}_1,\ldots, \boldsymbol{q}_r$, we arrive at the representation
$$
	L_{\delta}=\{\boldsymbol{p}^{\boldsymbol{Q}}\in T^r_{\boldsymbol{t}}: \gamma_{f_{\delta}}(\boldsymbol{p})\in \mathcal{L}(\boldsymbol{q}_1,\ldots, \boldsymbol{q}_r)\}.
$$	

Denoting by $L$ the union of the sets~$L_{\delta}$ over all faces~$\delta$ of the polyhedron~$\Delta_f$, we can finally formulate the main result of the paper.
\begin{theorem} 
	Let the Laurent polynomial~$f$ be nondegenerate for its Newton polyhedron~$\Delta_f$.
	Then the complete $\boldsymbol{Q}$-diagonal of the Laurent series~\eqref{eq:s0:1} of the rational function~$g(\boldsymbol{z})/f(\boldsymbol{z})$
	can be analytically continued along any path in $T^r_{\boldsymbol{t}} \setminus L$
	starting at the point~$\boldsymbol{t}_0\in \mathfrak{T}$.
   \label{thm1}
\end{theorem}

As was shown in works~\cite{Pochekutov21, Pochekutov22, Pochekutov25}, when studying diagonals it is convenient to pass to the complex torus $T^n_{\boldsymbol{w}}$ with coordinates $\boldsymbol{w}:=(w_1,\ldots, w_n)$,
which are related to $\boldsymbol{z}$ by the monomial transformation
\begin{equation}
	\label{eq:s0:tm}
		\boldsymbol{z}=\boldsymbol{w}^{\boldsymbol{A}}.
\end{equation}
Since the $r$-turple~$\boldsymbol{Q}$ generates an $r$-dimensional saturated sublattice of the lattice~$\mathbb{Z}^n$,
by means of the invariant factor theorem~\cite[theorem~16.6]{Curtis62} it can be constructively extended to a basis $\boldsymbol{q}_1,\ldots, \boldsymbol{q}_n$ of the entire lattice. Then the matrix~$\boldsymbol{A}$
is the inverse of the unimodular matrix~$\boldsymbol{B}$ whose columns are precisely the vectors $\boldsymbol{q}_1,\ldots, \boldsymbol{q}_n$.

The mapping~\eqref{eq:s0:tm} induces a one-to-one correspondence between Laurent polynomials~$f$ in the variables $\boldsymbol{z}$ and Laurent polynomials $\tilde{f}$ in the variables~$\boldsymbol{w}.$
The Newton polyhedra $\Delta_f$ and $\Delta_{\tilde{f}}$ are combinatorially equivalent under this correspondence: faces $\delta$ transform into faces~$\tilde{\delta}$ with the incidences between them preserved.
After the substitution, the matrix $M_{\delta}(\boldsymbol{\boldsymbol{z}})$ transforms into the matrix
$$
	\boldsymbol{M}_{\tilde{\delta}}(\boldsymbol{w}):=\left(
	\begin{array}{cccccc}
	w_1 \frac{\partial \tilde{f}_{\tilde{\delta}}}{\partial w_1} & \ldots & w_r \frac{\partial \tilde{f}_{\tilde{\delta}}}{\partial w_r} & w_{r+1}\frac{\partial \tilde{f}_{\tilde{\delta}}}{\partial w_{r+1}} & \ldots & w_n \frac{\partial \tilde{f}_{\tilde{\delta}}}{\partial w_n}  \\
	1 		& \ldots & 0 		& 0  		& \ldots & 0 \\
	\vdots 	& \ddots & \vdots 	& \vdots	& \ddots & \vdots\\ 
	0 		& \ldots & 1 		& 0 		& \ldots & 0\\
	\end{array}
	\right).
$$
Therefore, the set~$L$ can be described as follows. Consider $\tilde{f}(\boldsymbol{t},\boldsymbol{u})$ as a Laurent polynomial in the variables $\boldsymbol{u}\in T^s_{\boldsymbol{u}}$ with coefficients being Laurent polynomials in $\boldsymbol{t}\in T^r_{\boldsymbol{t}}$. Let $\Delta\subset \mathbb{R}^s$ be its Newton polyhedron. Denote by $\Sigma$ the family consisting of faces~$\sigma$ of the polyhedron~$\Delta$ for which the truncations $\tilde{f}_{\sigma}(\boldsymbol{t},\boldsymbol{u})$ do depend on $\boldsymbol{t}$. For a face~$\sigma\in \Sigma$, define a subset $L_{\sigma}$ in $T^r_{\boldsymbol{t}}$ such that for $\boldsymbol{t}\in L_{\sigma}$ there exists
a point $\boldsymbol{u}\in T^s_{\boldsymbol{u}}$ with the property
\begin{equation}
	\label{eq:s0:Ls}
		\tilde{f}_{\sigma}(\boldsymbol{t},\boldsymbol{u})=0,\ 
		\nabla_u \tilde{f}_{\sigma}(\boldsymbol{t},\boldsymbol{u}):=
		\left(\frac{\partial \tilde{f}_{\sigma}}{\partial u_1},\ldots,\frac{\partial \tilde{f}_{\sigma}}{\partial u_s}\right)(\boldsymbol{t},\boldsymbol{u})=0.
\end{equation}
In other words, $L_{\sigma}$ is the discriminant set of the algebraic equation $\tilde{f}_{\sigma}(\boldsymbol{t},\boldsymbol{u})=0$ in the unknowns $\boldsymbol{u}$ with coefficients being Laurent polynomials in~$\boldsymbol{t}$. Then $L$ is the union of the sets~$L_{\sigma}$ over all $\sigma\in \Sigma.$

\section{Integral Representations for the Complete Diagonal}

Let the Laurent series~\eqref{eq:s0:1} for a rational function $g(\boldsymbol{z})/f(\boldsymbol{z})$ converge in the domain $\Lambda^{-1}(E),$
where $E$ is a connected component of the complement~$\mathcal{A}^c_f$ of order~$\boldsymbol{\nu}$.

Consider a value $\boldsymbol{t}\in T^r_{\boldsymbol{t}}$ such that the amoebas of the binomials
$$f_1:=\boldsymbol{z}^{\boldsymbol{q}_1}-t_1 ,\ldots, f_r:=\boldsymbol{z}^{\boldsymbol{q}_r}-t_r$$ 
intersect at a point from the component~$E.$ For a vector $\boldsymbol{\varepsilon}:=(\varepsilon_1,\ldots,\varepsilon_r)$ with coordinates~$\pm 1$, denote by $\boldsymbol{x}_{\boldsymbol{\varepsilon}}(\boldsymbol{t})$ a point from the intersection $E_{\boldsymbol{\varepsilon}}$ of the component $E$ and the cone defined by the inequalities
$$
	\varepsilon_1\left(e^{ \boldsymbol{q}_1\cdot \boldsymbol{x} } - | t_1| \right) > 0,\ldots, 
	\varepsilon_r\left(e^{\boldsymbol{q}_r\cdot \boldsymbol{x}} - | t_r| \right) > 0, 
$$
where $\boldsymbol{x}:=(x_1,\ldots, x_n)$ are coordinates in~$\mathbb{R}^n$.
In other words, the amoebas $\mathcal{A}_{f_1},\ldots, \mathcal{A}_{f_r}$, which are hyperplanes in~$\mathbb{R}^n$ with normals $\boldsymbol{q_1},\ldots, \boldsymbol{q}_r$, respectively, for the specified choice of the parameter value~$\boldsymbol{t}$ divide the convex component~$E$ into $2^r$ open convex parts~$E_{\boldsymbol{\varepsilon}}$. The vector $\boldsymbol{\varepsilon}$ controls the choice of the point $\boldsymbol{x}_{\boldsymbol{\varepsilon}}(\boldsymbol{t})$ in a part of this partition (see Fig.~1, illustrating the case $n=3$ and $r=2$).

\begin{figure}[!h]
	\centering
	\includegraphics[scale=0.55]{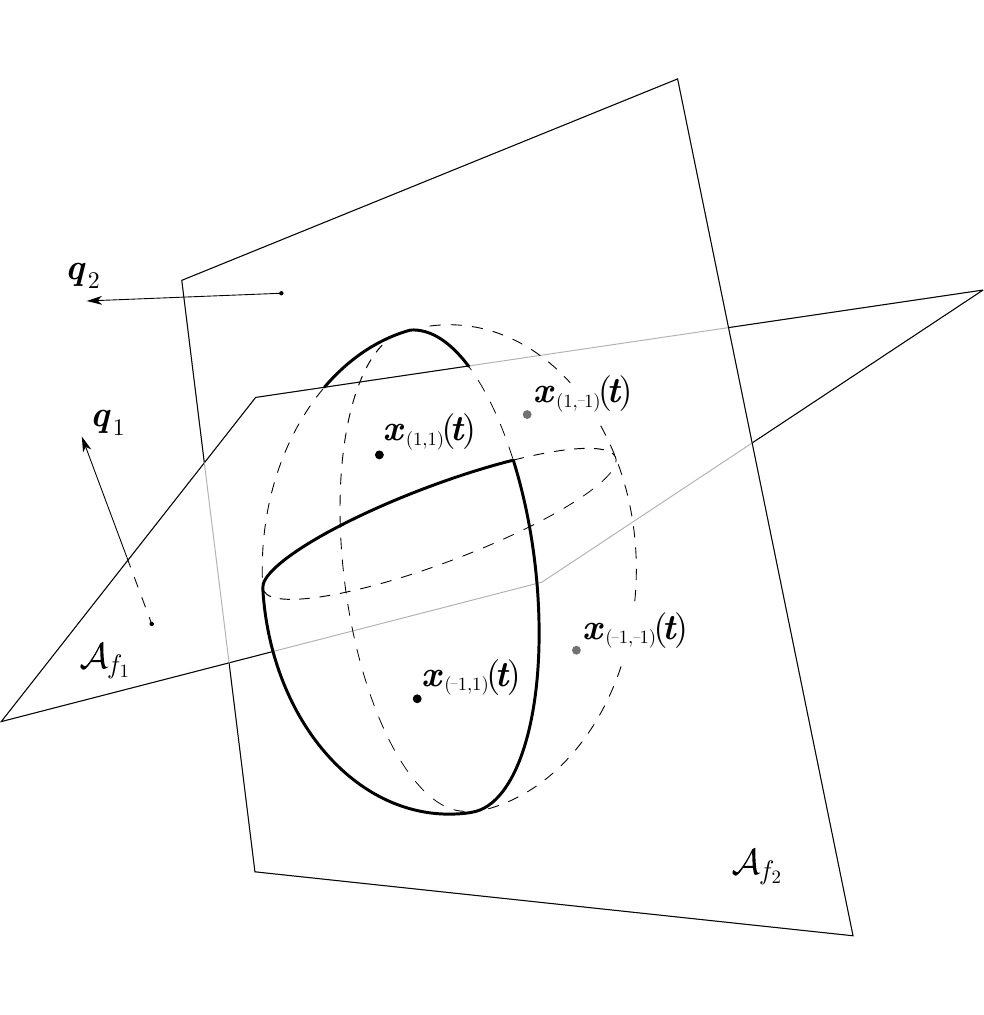}
	\caption{Partition of the component $E$ by the amoebas $\mathcal{A}_{f_1}, \mathcal{A}_{f_2}$ and the choice of points $\boldsymbol{x}_{\boldsymbol{\varepsilon}}(\boldsymbol{t})$.}
	\label{pic:1}
\end{figure}

The class of the real $n$-dimensional torus $\Gamma_{\boldsymbol{\varepsilon},\boldsymbol{t}}:=\Lambda^{-1}(\boldsymbol{x}_{\boldsymbol{\varepsilon}}(\boldsymbol{t}))$ in the homology group $H_n(T^n{\boldsymbol{z}}\setminus Z^\times(f\cdot f_1 \cdots f_r))$ with integer coefficients
does not depend on the choice of the point $\boldsymbol{x}_{\boldsymbol{\varepsilon}}(\boldsymbol{t})$ in~$E_{\boldsymbol{\varepsilon}}$ due to convexity. Then in the complement $T^n_{\boldsymbol{z}}\setminus Z^\times(f\cdot f_1 \cdots f_r)$ one can define the $n$-dimensional cycle 
$$
\Gamma_{\boldsymbol{t}} := \sum_{\boldsymbol{\varepsilon}} (\varepsilon_1\cdots \varepsilon_r) 
\Gamma_{\boldsymbol{\varepsilon},\boldsymbol{t}},
$$
where the summation is over all $\boldsymbol{\varepsilon}$ with coordinates $\pm 1.$

It is not difficult to show, following the reasoning from section~4 of article~\cite{Pochekutov09}, that the equality holds
\begin{equation}
\label{eq:s1:1}
	d_{\boldsymbol{Q}}(\boldsymbol{t})=\frac{1}{(2\pi \imath)^n} \int_{\Gamma_{\boldsymbol{t}}} 
	\frac{g(\boldsymbol{z})}{f(\boldsymbol{z})} \times \prod_{j=1}^{r} 
	\frac{ \boldsymbol{z}^{\boldsymbol{q}_j}}{\boldsymbol{z}^{\boldsymbol{q}_j}-t_j}
	\times \frac{dz_1\wedge\ldots\wedge d z_n}{z_1\cdots z_n}.
\end{equation}
Indeed, since the integral in \eqref{eq:s1:1} is the sum
$$
    \sum_{\boldsymbol{\varepsilon}}\int_{\Gamma_{\boldsymbol{\varepsilon},\boldsymbol{t}}} 
	\frac{g(\boldsymbol{z})}{f(\boldsymbol{z})} \times \prod_{j=1}^{r} 
	\frac{ \boldsymbol{z}^{\boldsymbol{q}_j}}{\boldsymbol{z}^{\boldsymbol{q}_j}-t_j}
	\times \frac{dz_1\wedge\ldots\wedge d z_n}{z_1\cdots z_n},
$$
then, expanding the factors of the integrand into Laurent series centered at the origin on the compacts~$\Gamma_{\boldsymbol{\varepsilon},\boldsymbol{t}},$ and then, using the commutativity of the series and integral over a compact set, using the multidimensional Cauchy formula we arrive at the right-hand side of expression~\eqref{eq:s0:2}.

As was shown in~\cite[Proposition~3]{Pochekutov22}, from the integral representation~\eqref{eq:s1:1}
using the monomial substitution~\eqref{eq:s0:tm}
and reducing the integration multiplicity, one can pass to a representation in the form of an $s$-dimensional integral
\begin{equation}
\label{eq:s1:2}
    d_{\boldsymbol{Q}}(\boldsymbol{t})=
   \int_{\tilde{\Gamma}_{\boldsymbol{t}}} 
   \omega(\boldsymbol{t}), 
\end{equation}
where the holomorphic in $T^s_{\boldsymbol{u}}\setminus Z^\times(\tilde{f}(\boldsymbol{t},\boldsymbol{u}))$ integrand differential form
$$
    \omega(\boldsymbol{t}):=
    \frac{1}{(2\pi\imath)^s} 
    \frac{\tilde{g}(\boldsymbol{t},\boldsymbol{u})}{\tilde{f}(\boldsymbol{t},\boldsymbol{u})} \frac{du_1\wedge\ldots \wedge du_s}{u_1\cdots u_s}
$$
depends rationally on the parameters~$\boldsymbol{t}=(t_1,\ldots, t_r)$.
The integration cycle $\tilde{\Gamma}_{\boldsymbol{t}}$ in~\eqref{eq:s1:2}
has the form $\Lambda^{-1}(\boldsymbol{y}(\boldsymbol{t})),$ where $\boldsymbol{y}(\boldsymbol{t})=(y_1(\boldsymbol{t}),\ldots, y_s(\boldsymbol{t}))$
is a point in the component $\tilde{E}'\subset \mathbb{R}^s$ of the complement of the amoeba of the polynomial $\tilde{f}(\boldsymbol{t},\boldsymbol{u}).$ More precisely,
according to~\cite[proposition~1]{Pochekutov22}, after the transformation~\eqref{eq:s0:tm} the component~$E$ of order $\boldsymbol{\nu}$ of the complement~$\mathcal{A}^c_f$ transforms
into the component~$\tilde{E}$ of order $\boldsymbol{\mu}=\boldsymbol{A} \boldsymbol{\nu}$ of the complement~$\mathcal{A}^c_{\tilde{f}}$. Then $\tilde{E}'$ is the component of order $(\mu_{r+1},\ldots, \mu_n).$

\section{The Set $L$ as a Landau Variety}

Note that among the faces $\sigma$ from the set $\Sigma$ defined in the second section, there may be vertices $\boldsymbol{\alpha}$ of the polyhedron $\Delta$. In this case, the truncations $\tilde{f}_{\boldsymbol{\alpha}}(\boldsymbol{t},\boldsymbol{u})$ are monomials of the form $P_{\boldsymbol{\alpha}}(\boldsymbol{t}) \boldsymbol{u}^{\boldsymbol{\alpha}}.$ Therefore, by removing from the torus $T^r_{\boldsymbol{t}}$ the union of all sets $L_{\boldsymbol{\alpha}},$
we get that for the remaining distinct, fixed values of $\boldsymbol{t}$, all polynomials $\tilde{f}(\boldsymbol{t},\boldsymbol{u})$ in the variables $\boldsymbol{u}$ (with coefficients in $\mathbb{C}$) have the same Newton polyhedra $\Delta$.

Furthermore, after excluding from $T^r_{\boldsymbol{t}}$ the sets $L_{\sigma}$ for faces from $\Sigma$ other than vertices, we get that for the remaining values of $\boldsymbol{t}$, the truncations $\tilde{f}_{\sigma}(\boldsymbol{t}, \boldsymbol{u})$ satisfy the conditions from the definition of a polynomial being nondegenerate for its Newton polyhedron.

At the same time, the truncation $\tilde{f}_{\delta}(\boldsymbol{t},\boldsymbol{u})$ to a face $\delta\not \in \Sigma$ depends only on $\boldsymbol{u}$; it corresponds bijectively to the truncation of the polynomial $\tilde{f}$ to some face $\gamma$ of its Newton polyhedron $\Delta_{\tilde{f}}$. Note that a Laurent polynomial $f$ nondegenerate for its Newton polyhedron $\Delta_f$ transforms, after the substitution \eqref{eq:s0:tm}, into a Laurent polynomial $\tilde{f}$ nondegenerate for its Newton polyhedron $\Delta_{\tilde{f}}$. Therefore, for all $\boldsymbol{t}\in T^r_{\boldsymbol{t}}$
and at all points $\boldsymbol{p} \in Z^{\times}(\tilde{f}_{\sigma}(\boldsymbol{t},\boldsymbol{u}))$, the gradient of the truncation
$\nabla{\boldsymbol{u}} \tilde{f}_{\sigma}(\boldsymbol{t},\boldsymbol{p}) \neq \boldsymbol{0}.$ Consequently, for a fixed $\boldsymbol{t}$ from the set
$$ T:= T^r_{\boldsymbol{t}} \setminus \bigcup_{\sigma\in \Sigma} L_{\sigma}$$
the Laurent polynomial $\tilde{f}(\boldsymbol{t},\boldsymbol{u})$ in the variables $\boldsymbol{u}$
is nondegenerate for its Newton polyhedron $\Delta$.

We will consider the smooth toric compactification $\mathbb{X}_{\mathcal{F}}$ of the complex torus $T^s_{\boldsymbol{u}}$, constructed from a sufficiently fine refinement of the fan $\mathcal{F}$ dual to $\Delta$ (see [18, §2]).  Let  $l_1,\ldots, l_d$ be $(s-1)$-dimensional orbits in $\mathbb{X}_{\mathcal{F}}$ corresponding  to one-dimensional generators of the fan $\mathcal{F}$.  We define the hypersurfaces in $\mathbb{X}_{\mathcal{F}}\times T$:
$$
	\mathcal{S}_0:=\bigcup_{\boldsymbol{t}\in T} \overline{Z^{\times}(\tilde{f}(\boldsymbol{t},\boldsymbol{u}))},\ 
	\mathcal{S}_j:=l_j\times T, \ j=  1,\ldots,d,
$$
where  the bar denotes closure in $\mathbb{X}_{\mathcal{F}}$. By the resolution of singularities theorem~\cite[s.~2]{Khovanskii77}, the hypersurface $\overline{Z^{\times}(\tilde{f}(\boldsymbol{t},\boldsymbol{u}))}$ intersects transversely
the hyperplanes at infinity $l_1,\ldots, l_d$ in $\mathbb{X}_{\mathcal{F}}$. Therefore, the hypersurfaces
$\mathcal{S}_0,\mathcal{S}_1,\ldots, \mathcal{S}_d$ are in general position in $\mathbb{X}_{\mathcal{F}}\times T$.
We stratify their union $\mathcal{S}$ by all possible intersections, i.e.,
the strata are  of the form
$$
	A_i:= \mathcal{S}_i \setminus \bigcup_{j\neq i } \mathcal{S}_j,\
	A_{ij}:= \mathcal{S}_i \cap \mathcal{S}_j \setminus \bigcup_{k\neq i,j} \mathcal{S}_k,\ \ldots\ .
$$

Consider the projection $\pi: \mathbb{X}_{\mathcal{F}}\times T\to T$; it is a proper map
since $\mathbb{X}_{\mathcal{F}}$ is compact. For each stratum $A$, denote by $c A$ the subset of this stratum consisting of points where
the rank of $d \pi |_{A}$ is less than $r.$ Note that the union $c\mathcal{S}$ over all strata $A$ of the sets $c A$ is closed \cite[s.~IV.4.4]{Pham11}.

The closure $\overline{c A}$ and its image $\pi (\overline{c A})$,
called the \textit{Landau variety of the stratum}~$A$,
are complex analytic sets \cite[s.~IV.5.2]{Pham11}.
Since the union $c\mathcal{S}$ is closed, $\pi(c \mathcal{S})$
is also a complex analytic set by Remmert's proper mapping theorem;
it is called the \textit{Landau variety of the stratified set}~$\mathcal{S}$.

Direct computation shows that the sets $cA_{i_1\ldots i_k}$ are empty if all $i_j>0.$

The set $cA_{0}$ consists of points $(\boldsymbol{t}, \boldsymbol{u})\in T^r_{\boldsymbol{t}}\times T^s_{\boldsymbol{u}}$ satisfying the system of algebraic equations
$$
 	\left\{
 	\begin{array}{rcr}
 		\tilde{f}(\boldsymbol{t},\boldsymbol{u}) & = & 0, \\
 		\nabla_{\boldsymbol{u}} \tilde{f}(\boldsymbol{t},\boldsymbol{u}) & = & 0.\\
 	\end{array}
 	\right.
$$
Since $\tilde{f}(\boldsymbol{t},\boldsymbol{u})=\tilde{f}_{\Delta}(\boldsymbol{t},\boldsymbol{u})$, the image $\pi(c A_0)$ coincides with the set $L_{\Delta}$. Moreover, the entire Landau variety
of the set $\mathcal{S}$ is represented in the form
$$
	L=\bigcup_{\sigma\in \Sigma} L_{\sigma}. 
$$
Indeed, each truncation $\tilde{f}_{\tau}(\boldsymbol{t},\boldsymbol{u})$ to a face $\sigma \in \Sigma$ of codimension $k$ can be written as a polynomial in the variables $v_1,\ldots, v_{s-k}$ 
by means of a suitable monomial change $\boldsymbol{u}=\boldsymbol{v}^{\boldsymbol{C}}$, where $\boldsymbol{C}$ is a unimodular matrix. Then the zero set of this polynomial in the torus $T^s_{\boldsymbol{v}}$ describes a stratum of the form $A_{0i_1\ldots i_k}$. Conversely, to each nonempty stratum of the form $A_{0i_1\ldots i_k}$, one can associate a truncation $\tilde{f}_{\sigma}(\boldsymbol{t},\boldsymbol{u})$.

Thus, the set $L=\pi(c \mathcal{S})$. Then the projection
\begin{equation}
\label{eq:pi}
\pi: (\mathbb{X}_{\mathcal{F}}\times T, \mathcal{S}) \to T
\end{equation}
defines a locally trivial fibration of a pair  by Thom's isotopy theorem since $\pi$ is proper.

\section{Proof of Theorem 1}
\begin{proof}
The locally trivial fibration~\eqref{eq:pi} induces a locally trivial
homological fibration with the fiber over a point $\boldsymbol{t}\in T$ equals to
$G(\boldsymbol{t}):=H_s(\mathbb{X}_{\mathcal{F}}\setminus U(\boldsymbol{t})),$
where
$$U(\boldsymbol{t}):=\overline{Z^{\times}(\tilde{f}(\boldsymbol{t},\boldsymbol{u}))}\cup l_1\ldots \cup l_d.$$
Since $ 	\mathbb{X}_{\mathcal{F}}\setminus U(\boldsymbol{t}) = 
	T^s_{\boldsymbol{u}}\setminus 
	Z^\times(\tilde{f}(\boldsymbol{t},\boldsymbol{u})),$
the fiber $G(\boldsymbol{t})$ coincides with the homology group to which the integration cycle class in~\eqref{eq:s1:2} belongs.

The set $T$ is obtained by removing from $\mathbb{C}^s$ complex analytic sets of positive codimension, therefore it is linearly connected. Suppose that for a point $\boldsymbol{t}_0$ the representation~\eqref{eq:s1:2} holds; choose an arbitrary path $\alpha:[0;1]\to T$,
connecting the points $\boldsymbol{t}_0$ and $\boldsymbol{t}_1$. For each point of the path $\alpha(\tau)$ one can choose an open ball
$B(\alpha(\tau))$ centered at this point, with a sufficiently small radius so that it lies in $T$.
The considered homological fibration will be trivial over $B(\alpha(\tau))$ since the ball is contractible  . Therefore, the groups $G(\boldsymbol{t})$
will be isomorphic for all $\boldsymbol{t}\in B(\alpha(\tau))$.
Moreover, for any continuous family of classes $[\Gamma_{\boldsymbol{t}}]$ from $G(\boldsymbol{t}),$
parameterized by $\boldsymbol{t}\in B(\alpha(\tau))$, the cycle class $\Gamma_{\alpha(\tau)}$
in $G(\boldsymbol{t})$ will coincide with $[\Gamma_{\boldsymbol{t}}]$ for all $\boldsymbol{t}\in B(\alpha(\tau)).$ By choosing a finite cover of the support of the path $\alpha$ by balls $B(\alpha(\tau))$ for compactness reasons, where
$$
	\tau_0:=0 < \tau_1 < \ldots < \tau_{N-1} < \tau_N:=1,
$$
we are able to continue the integral~\eqref{eq:s1:2} from a neighborhood in $B(\boldsymbol{t}_0)$ to $B(\boldsymbol{t}_1).$

First, for $\boldsymbol{t}\in B(\boldsymbol{t}_0)$, we replace the integration cycle in~\eqref{eq:s1:2} with the cycle $\Gamma_{\boldsymbol{t}_0}$.
Since $\Gamma_{\boldsymbol{t}_0}$ is compact, the integral $\int_{\Gamma_{\boldsymbol{t}_0}}\omega(\boldsymbol{t})$
defines an analytic function in $B(\boldsymbol{t}_0)$. Its analytic continuation into $B(\alpha(\tau_1))$ has the form
$\int_{\Gamma_{\alpha(\tau_1)}} \omega(\boldsymbol{t})$. Continuing this process, after a finite number of steps we obtain the analytic continuation
of the original element into $B(\boldsymbol{t}_1)$ as the integral $\int_{\Gamma_{\boldsymbol{t}_1}} \omega(\boldsymbol{t})$, where
$\Gamma_{\boldsymbol{t}_1}$ is the result of the stepwise deformation of the cycle $\Gamma_{\boldsymbol{t}_0}$.

Note that the result of the analytic continuation will not change if the path $\alpha$ is replaced by a path $\alpha'$ homotopic to it
in $T$, starting at the point $\boldsymbol{t}_0$ and ending at the point $\boldsymbol{t}_1$.
\end{proof}

\section{Examples of Hypergeometric Diagonals}

\textbf{Example 1.}
Consider the generalized hypergeometric function
$$
		{}_lF_{l-1}(\tfrac{1}{l+1},\ldots, \tfrac{l}{l+1};\tfrac{1}{l},\ldots, \tfrac{l-1}{l}; \tfrac{(l+1)^{l+1}}{l^l}t^l),\ l=1,2,\ldots , 
	$$
which is the $(1,1,1)$-diagonal of the Taylor expansion of the rational function
$1/f(z_1,z_2,z_3)$, holomorphic at the point $(0,0,0)$, where the polynomial is
$$	f(z_1,z_2,z_3):=1-z_1-(z_2z_3)^l-z_3.$$	
Using transformation~\eqref{eq:s0:tm}, which in this case has components
$z_1=\frac{w_1}{w_2 w_3},\ z_2=w_2,\ z_3=w_3,$
we pass to the Laurent polynomial
$$ \tilde{f}(t, u_1, u_2)=1-\frac{t}{u_1 u_2}-u^l_1 u_2^l-u_2.$$
The only non-empty set $L_\sigma,$ corresponds to the truncation $\tilde{f}_{\sigma}=1-t/(u_1 u_2)-u^l_1 u_2^l,$
for which system~\eqref{eq:s0:Ls} is equivalent to the system
$$ 
	1-\frac{t}{u_1 u_2}-u^l_1 u_2^l=\frac{t}{u_1 u_2} -l u_1^{l} u_2^l=0.
$$
Eliminating the variables $u_1, u_2$, we find that the set $L_{\sigma}$ consists of a single point
$$ t_0=\frac{l}{l+1}\left(\frac{1}{l+1}\right)^{\tfrac1l},$$
in complete agreement with known results on the singularities of the generalized hypergeometric function.

\begin{figure}
\centering
\includegraphics[width=0.48\textwidth]{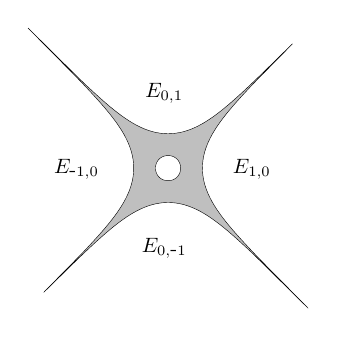}
\includegraphics[width=0.48\textwidth]{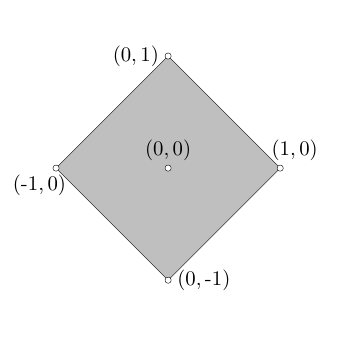}
\caption{Amoeba (gray, left), its complement components (the bounded component without label is $E_{0,0}$)
  and the Newton polytope (right) of the Laurent polynomial $1-\frac{t_1}{u_1}-u_1-\frac{t_2}{u_2} -u_2$.}
\label{fig:fig2}
\end{figure}

\noindent \textbf{Example 2.} Consider the Appell function of the fourth kind
$$
		F_4(1,\tfrac{1}{2};1,1; \tfrac{t_1}{4},\tfrac{t_2}{4}),
$$
which is a hypergeometric series in two variables $t_1,t_2,$ converging in the domain $|t_1|^{\tfrac12}+|t_2|^{\tfrac12}<2$. It is the $(\boldsymbol{q}_1,\boldsymbol{q}_2)$-diagonal
of the Taylor expansion of the rational function $1/f(\boldsymbol{z})=(1-z_1-z_2-z_3-z_4)^{-1}$ for $\boldsymbol{q}_1=(1,1,0,0)$
and $\boldsymbol{q}_2=(0,0,1,1).$ Using a suitable transformation~\eqref{eq:s0:tm}, we can
pass to the Laurent polynomial
 $$
		\tilde{f}(t_1,t_2,u_1,u_2)=1-\frac{t_1}{u_1}-u_1-\frac{t_2}{u_2}-u_2,
$$
whose amoeba for small absolute values of the parameters $t_1,t_2$ and Newton polytope are shown in Figure~2 .
The Landau variety $L$ consists solely of the set $L_\Delta.$ Eliminating the variables $u_1,u_2$ from system~\eqref{eq:s0:Ls}, which in this case takes the form
   $$
		1-\frac{t_1}{u_1}-u_1-\frac{t_2}{u_2}-u_2=\frac{t_1}{u_1^2}-1=\frac{t_2}{u_2^2}-1=0,
   $$
  we find that $L_\Delta =\{ (t_1,t_2)\in (\mathbb{C}^{\times})^2:16t_1^2-32t_1t_2+16t_2^2-8t_1-8t_2+1=0\}.$

  Therefore, according to Theorem~1, the function $F_4(1,\tfrac{1}{2};1,1; \tfrac{t_1}{4},\tfrac{t_2}{4})$ continues to the domain
$$\mathbb{C}^2\setminus \{t_1=0\}\cup\{t_2=0\}\cup L_{\Delta}. $$

\bigskip
\bibliographystyle{unsrt}  


\end{document}